%% file: main__3_.tex
\documentclass[11pt, hidelinks]{article}
\input{preamble}
\usepackage{optidef}
\usepackage{xcolor}

\title{Exact Theory and Optimal Stepsizes for Stochastic Gradient Descent in Smooth Convex Minimization}
\author{
    Benjamin Grimmer\footnote{Johns Hopkins University, Department of Applied Mathematics and Statistics, \texttt{grimmer@jhu.edu}}
    \and
    Qining Shen\footnote{Johns Hopkins University, Department of Applied Mathematics and Statistics, \texttt{qshen18@jh.edu}}
}
\date{}

\begin{document}

\maketitle

\begin{abstract}
    This work considers the design of stepsize schedules for stochastic gradient descent in smooth convex optimization. Throughout, we assume that gradient noise is conditionally zero-mean with bounded variance. Among problems with dimension $d\geq N+1$, we provide an exact worst-case analysis of the expected suboptimality of stochastic gradient descent's final iterate. Our analysis applies to any (potentially nonconstant) schedule of $N$ stepsizes in $[0,1.5]^N$. Here, exact means that we provide a matching objective function and noise distributions where our guarantee is attained. Whenever the variance bound is at least $\Omega(1/\sqrt{N})$, our theory provides the minimax optimal stepsize schedule for stochastic gradient descent and proves its uniqueness. These optimal stepsizes monotonically decrease, ruling out any accelerations in expected final suboptimality from fixed fractal schedules observed in the deterministic case.
\end{abstract}

\section{Introduction}
We consider stochastic smooth convex minimization problems of the form
\begin{equation}
    f_\star = \min_{x\in\mathbb{R}^d} f(x)
\end{equation}
in dimension $d$. Throughout, we assume the objective is $L$-smooth (i.e., has an $L$-Lipschitz gradient), that this minimum value is attained at some $x_\star$, and that an initialization $x_0$ is provided with bounded distance $\|x_0-x_\star\|\leq D$.

For any such problem instance, we study Stochastic Gradient Descent (SGD) with a fixed iteration budget of $N$ steps. At each iteration $n=1,\dots,N$, the algorithm is provided a stochastic estimate of the objective's gradient, denoted
\begin{equation*}
    \widetilde g_{n-1} = \nabla f(x_{n-1}) + \epsilon_{n-1}
\end{equation*}
with $\epsilon_{n-1}$ sampled from a conditional distribution $P_{n-1}(\cdot \mid \mathcal{F}_{n-1})$, conditioned on the history $\mathcal{F}_k=\sigma(\epsilon_0,\dots,\epsilon_{k-1})$.  We assume this noise is conditionally zero-mean and has variance at most $\sigma^2$, conditioned on the history. That is, almost surely,
\begin{equation*}
    \mathbb{E}[\epsilon_{n-1} \mid \mathcal{F}_{n-1}]=0 \quad \text{and} \quad \mathbb{E}[\|\epsilon_{n-1}\|^2 \mid \mathcal{F}_{n-1} ]\leq \sigma^2.
\end{equation*}
SGD then uses this stochastic oracle and a fixed sequence of $N$ stepsizes $h_{0},\dots,h_{N-1} \in\mathbb{R}$ to iterate for $n=1,\dots,N$ with
\begin{equation} \label{eq:SGD}
    x_n = x_{n-1} - \frac{h_{n-1}}{L}\widetilde g_{n-1}.
\end{equation}
Afterwards, the algorithm outputs its final iterate $x_N$. Note that many variations of this method have been proposed, including various momentum, averaging, and extrapolation dynamics. Here we study the core method in this simplest, fundamental form.

A problem instance in dimension $d$ is defined by the initialization $x_0\in\mathbb{R}^d$, objective function $f\colon \mathbb{R}^d\rightarrow \mathbb{R}$ with minimizer $x_\star\in\mathbb{R}^d$, and sequence of distributions from which noise is drawn, $P_0,\dots,P_{N-1}$. We denote these collectively by some
\begin{equation*}
    \mathtt{p} = (x_0,f, P_0,\dots,P_{N-1}).
\end{equation*}
We denote the set of all problem instances by $\mathtt{P}(L,D,\sigma^2)$, constrained by $L$-smoothness, $D$ initial distance, and conditional $\sigma^2$ variance.

An $N$-step stochastic gradient descent method is defined once one specifies its vector of stepsizes $h\in\mathbb{R}^N$. 
For a fixed vector of stepsizes, we measure the performance of SGD by the expected final objective gap, $\mathbb{E}_{\epsilon} [f(x_N)] - f_\star$. Notice the only source of randomness considered is the noise at each iteration, so this expectation is only over all $\epsilon_i$.

Our main goal is to characterize the exact worst-case last-iterate performance of SGD, for fixed vectors of stepsizes $h\in\mathbb{R}^N$, over this problem class, for fixed $L,D,\sigma^2$. This is known as the algorithm's ``Performance Estimation Problem'' (PEP). Formally, we denote this by
\begin{equation}\label{eq:PEP}
    \mathtt{PEP}(h,L,D,\sigma^2) = \max_{\mathtt{p} \in \mathtt{P}(L,D,\sigma^2)} \left\{\mathbb{E}_{\epsilon} \left[f(x_N)\right] - f_\star \right\}.
\end{equation}
Here we consider a fixed iteration budget $N$ and problem dimension $d$, which are suppressed in the above notation.

PEPs were first proposed and exactly analyzed by the works~\cite{drori2014performance,taylor2017smooth,taylor2017composite} in the context of deterministic optimization. For sufficiently structured deterministic problem classes (e.g., $L$-smooth convex minimization), provided $d$ is large enough, the related PEP can often be described as a semidefinite program (SDP). From this, computer-assisted design techniques can then be applied.

The discovery of our results critically relied on such numerical and analytic insights to develop intuitions and formulate conjectures. In particular, the stochastic PEP approach of~\cite{rubbens2026computeraidedanalysesstochasticfirstorder} applies to the considered setting and informed our preliminary study of SGD. We present the resulting arguments directly, so our main results can be read without familiarity with PEP formulations.

\paragraph{Our Contributions.} 
Our main result exactly solves the above PEP for SGD with any schedule of stepsizes $h\in\mathbb{R}^N$ in the short step regime $[0,1.5]^N$, provided $d\geq N+1$. In short, we find that
\begin{equation} \label{eq:pep-value}
    \mathtt{PEP}(h,L,D,\sigma^2) = \frac{LD^2}{
        2\left(1+2R_0\right)
        }
        +
        \frac{\sigma^2}{2L}
        \sum_{j=0}^{N-1}
        \frac{h_j^2}{
        1+2R_{j+1}
        }
\end{equation}
where $R_n = \sum_{i=n}^{N-1} h_i$ denotes the total remaining stepsize length at iteration $n$ with $R_N=0$. Specializing to $\sigma=0$, the first term above provides the exact worst-case of deterministic gradient descent in the short step regime---matching the numerically supported conjecture of~\cite[Section 4.4.1]{Daccache2019}. For example,~\eqref{eq:pep-value} recovers the deterministic constant stepsize $h\in [0,1]$ analysis of~\cite{drori2014performance}. The second term above then captures the effect of the noise at each iteration $j$ on the final iterate.

In particular, our convergence upper bound applies for problems in any dimension, only requiring the short step assumption, and our convergence lower bound applies for any nonnegative stepsizes, only requiring the large dimension assumption. When both hold, the two bounds agree, providing the above exact worst-case characterization. The following theorem formalizes this.

\begin{theorem} \label{thm:main-rate}
    For any iteration budget $N\geq 1$, stepsizes $h\in [0,1.5]^N$, and problem parameters $L,D>0$, $\sigma^2\geq 0$, the SGD method~\eqref{eq:SGD} applied to any problem instance $(x_0,f,P_0,\dots,P_{N-1})\in\mathtt{P}(L,D,\sigma^2)$ has
    \begin{equation*}
        \mathbb{E}_{\epsilon} \left[f(x_N)\right] - f_\star \leq \frac{LD^2}{
        2\left(1+2R_0\right)
        }
        +
        \frac{\sigma^2}{2L}
        \sum_{j=0}^{N-1}
        \frac{h_j^2}{
        1+2R_{j+1}
        }.
    \end{equation*}
    Moreover, for any stepsizes $h\in \mathbb{R}_+^N$, if $d\geq N+1$, there exists a problem instance with
    \begin{equation*}
        \mathbb{E}_{\epsilon} \left[f(x_N)\right] - f_\star = \frac{LD^2}{
        2\left(1+2R_0\right)
        }
        +
        \frac{\sigma^2}{2L}
        \sum_{j=0}^{N-1}
        \frac{h_j^2}{
        1+2R_{j+1}
        }.
    \end{equation*}
\end{theorem}
Note that the high-dimensionality assumption $d\geq N+1$ is standard in the dimension-independent analysis of first-order methods~\cite{nemirovsky1983problem}. As mentioned above, this assumption is also common within the PEP literature.

Section~\ref{sec:UB} proves the claimed upper bound on final expected suboptimality via a short inductive argument. There we provide an interpretation of this tight proof as iteratively building expected suboptimality guarantees on a sequence of Moreau envelopes.
Section~\ref{sec:LB} proves the matching lower bound in our main theorem by constructing an attaining worst-case objective and collection of distributions. In fact, there we show a stronger result: every realization of the noise vectors $\epsilon_0,\dots,\epsilon_{N-1}$ on our constructed worst-case problem instance attains the above equality.

\subsection{A Corollary: The Exactly Optimal Stepsize Schedule}
We next optimize this bound over fixed schedules of stepsizes when $\sigma^2>0$. Consider
\begin{equation} \label{eq:minimax-design}
    r_\star(L,D,\sigma^2) = \min_{h\in\mathbb{R}^N} \max_{\mathtt{p} \in \mathtt{P}(L,D,\sigma^2)} \left\{\mathbb{E}_{\epsilon} \left[f(x_N)\right] - f_\star \right\}
\end{equation}
where we continue to suppress the dependence on the fixed values $N,d$. The above value $r_\star(L,D,\sigma^2)$ denotes the best-possible final guarantee any fixed schedule can provide. Stepsize schedules that attain the above minimum are called minimax optimal. 

The nonnegative restriction of our main theorem's lower bounding statement is for ease in stating the result. If $d\geq N+1$, Section~\ref{sec:negative-steps} shows that the worst-case performance of any schedule $h$ (allowing negative stepsizes) is at least the right-hand-side of~\eqref{eq:pep-value} evaluated at $h^+=(\max\{h_i,0\})_{i=0}^{N-1}$. Hence, one has
\begin{equation} \label{eq:minimized-LB}
    r_\star(L,D,\sigma^2) \geq \min_{h\in\mathbb{R}^N_+} \left\{\frac{LD^2}{
        2\left(1+2R_0\right)
        }
        +
        \frac{\sigma^2}{2L}
        \sum_{j=0}^{N-1}
        \frac{h_j^2}{
        1+2R_{j+1}
        }\right\}.
\end{equation}
This minimization problem can be solved by observing that the stated objective is strictly convex in $h$ and possesses a unique stationary point in the interior of $\mathbb{R}^N_+$. In particular, the minimizing stepsize schedule is given by a system of equations setting the gradient equal to zero,
\begin{equation}\label{eq:stationarity-stepsizes}
    \nabla \left(\frac{LD^2}{
        2\left(1+2R_0\right)
        }
        +
        \frac{\sigma^2}{2L}
        \sum_{j=0}^{N-1}
        \frac{h_j^2}{
        1+2R_{j+1}
        }\right) (h^\star)= 0.
\end{equation}

To solve this system, consider the recurrence $c_{i+1}=c_i+c_i^2$ for $i=0,\dots,N-1$, setting $c_0$ as the unique positive solution to the equation
\begin{equation*}
c_0\prod_{i=0}^{N-1} \left(1+2c_i\right)^2 = \frac{L^2D^2}{\sigma^2}.
\end{equation*}
We claim the following stepsize schedule solves the necessary and sufficient optimality condition~\eqref{eq:stationarity-stepsizes}
\begin{align}
h_i^\star = c_i \prod_{k=i+1}^{N-1}\left(1+2c_k\right), \qquad i=0,\dots,N-1. \label{eq:optimal-steps}
\end{align}
Appendix~\ref{app:optimal-stepsizes-verification} verifies this zero gradient condition.
Plugging these optimized stepsizes in provides a lower bound on the performance of SGD with any stepsize schedule. Namely, for any $L,D,\sigma^2>0$,
\begin{equation} \label{eq:general-LB}
    r_\star(L,D,\sigma^2) \geq  \sigma D \sqrt{c_0} - \frac{\sigma^2}{2L}c_N.
\end{equation}

Often these optimized stepsizes lie in the range $[0,1.5]^N$ where our matching upper bound theory applies. In particular, whenever the variance bound is not too small, $\sigma^2 \geq \frac{2}{3\sqrt{3}}\frac{L^2D^2}{\sqrt{N}}$, one has that $h^\star \in [0,1.5]^N$. Appendix~\ref{app:optimal-stepsizes-range} verifies this claimed bound.

Thus, given the preceding sufficient variance condition, the above minimization of our lower bound is attained at a short step schedule where the lower bound is tight. Hence, in this regime, $h^\star$ is the minimax optimal schedule. That is, no alternative fixed schedule of stepsizes $h$ can provide a stronger worst-case guarantee. The fact that $h^\star$ is the unique solution to~\eqref{eq:stationarity-stepsizes} and is strictly positive establishes it as the unique minimax optimal stepsize schedule.

The following result formally states this corollary.
\begin{corollary}
    Provided $\sigma^2 \geq \frac{2}{3\sqrt{3}}\frac{L^2D^2}{\sqrt{N}}$ and $d\geq N+1$, the stepsize schedule $h^\star$ in~\eqref{eq:optimal-steps} is exactly minimax optimal, i.e., solves~\eqref{eq:minimax-design}, and uniquely so. In particular, SGD with stepsizes $h^\star$ has
    \begin{equation} \label{eq:optimal-rate}
        \mathbb{E}_{\epsilon} \left[f(x_N)\right] - f_\star \leq r_\star(L,D,\sigma^2) = \sigma D \sqrt{c_0} - \frac{\sigma^2}{2L}c_N = \frac{\sigma D}{\sqrt{N}} - \frac{\sigma^{3/2}\sqrt{D}}{\sqrt{L}N^{3/4}} + o(1/N^{3/4}).
    \end{equation}
    No other fixed schedule of stepsizes has the same or smaller worst-case final objective gap.
\end{corollary}

Figure~\ref{fig:stepsizes} shows the optimal stepsize schedules for $N=10$ and $N=100$ steps with variance bounds $\sigma^2=1/8,1/2,2,8$, fixing $L=D=1$. One can see the optimal stepsizes decay almost linearly. Indeed, one can approximate the optimal schedule, given $N,L,D,\sigma^2$, by the simpler formula
\begin{equation} \label{eq:simple-asym-optimal-formula}
    \tilde{h}_i = \sqrt{\frac{LD}{\sigma} \frac{1}{N^{3/2}}} + \frac{LD}{\sigma} \frac{N-1-i}{N^{3/2}}.
\end{equation}
Plugging these simplified stepsizes into Theorem~\ref{thm:main-rate} gives a sharp convergence rate of the same order as~\eqref{eq:optimal-rate}. For fixed $L,D,\sigma^2$, as $N\rightarrow \infty$, the resulting guarantee differs from the minimax optimal rate only in the little-o term.

\begin{figure}
    \centering
    \includegraphics[width=\linewidth]{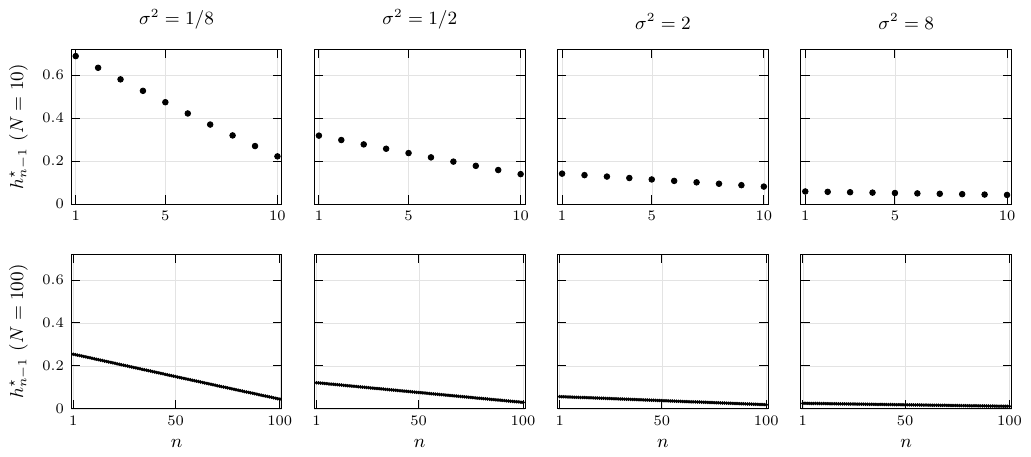}
    \caption{The exactly minimax optimal stepsize schedule $h^\star$ from~\eqref{eq:optimal-steps} for varied budgets $N$ and variances $\sigma^2$, fixing $L=D=1$. Asymptotically, this schedule behaves like the affine function~\eqref{eq:simple-asym-optimal-formula}.}
    \label{fig:stepsizes}
\end{figure}

\subsection{Related Work} \label{sec:related}

The worst-case performance of stochastic gradient descent methods has been studied widely~\cite{moulines2011nonasymptotic,shamir2013stochastic,jain2019making,liu2026revisiting}. Optimal last-iterate guarantees and linearly decaying stepsizes in the related setting of bounded stochastic subgradient methods were previously observed in~\cite{defazio2024lineardecay}. That work built upon the exactly optimal, linearly decaying stepsizes for the deterministic subgradient method of~\cite{zamani2025exact}.

Within the bounded-variance stochastic gradient descent model we consider, the results~\cite[Theorem~3.4 and Lemma~4.2]{liu2026revisiting} provided last-iterate guarantees when $h_{n-1} = a(N-n+1)/N^{3/2}$ with $0<a\leq \sqrt{N}/2$. They showed a convergence guarantee of
\begin{equation*}
    \mathbb{E}_{\epsilon}\left[f(x_N)\right]-f_\star
    \leq \frac{LD^2}{a\sqrt{N}}
    +\frac{4a\sigma^2}{L\sqrt{N}}.
\end{equation*}

Note that this bound is not tight, as applying Theorem~\ref{thm:main-rate} provides a constant factor improvement in this rate for all such $a$.
Provided $N\geq L^2D^2/\sigma^2$, their bound above is minimized by setting $a=LD/(2\sigma)$, where their optimized guarantee equals $4\sigma D/\sqrt{N}$. In contrast, our minimax optimal convergence guarantee~\eqref{eq:optimal-rate} improves this leading constant factor from $4$ to $1$ if the minimax optimal schedule above is used.

In deterministic smooth convex optimization, the series of works~\cite{dasgupta2024branchbound,grimmer2024provably,altschuler2025silver} identified faster worst-case convergence guarantees from nonmonotone stepsize schedules with long steps, including fractal schedules. Building on the recursive gluing technique of~\cite{altschuler2025silver}, the concurrently developed composition theories of~\cite{grimmer2025composing,zhang2026concatenation} provided equivalent constructions of schedules---called the Optimized Basic f-composable Schedules (OBS-F) in~\cite{grimmer2025composing}---conjectured to be exactly minimax optimal when $\sigma=0$. Our results here establish optimality and uniqueness for a simpler monotonic schedule when $\sigma^2 \geq \frac{2}{3\sqrt{3}}\frac{L^2D^2}{\sqrt{N}}$ and $d\geq N+1$. Thus, nonmonotone fractal stepsizes offer no improvement in worst-case guarantees once variance bounds are moderate.
Complementing this,~\cite{vernimmen2026robustness} studied the sensitivity of long-step schedules to relative gradient errors and proposed robust variants.

Exact analyses of stochastic first-order methods have been facilitated by the PEP techniques, first developed by~\cite{drori2014performance,taylor2017composite,taylor2017smooth}. For settings of finite sum minimization, where stochasticity arises from sampling indices, these tools can be applied by enumerating possible sampling outcomes. The works of~\cite{taylor2019stochastic,cortild2026biasoptimal} provided guarantees for stochastic gradient descent through computer-assisted potential and Lyapunov analyses based on such an approach. A more general approach to stochastic PEP modeling with relaxations based on expected interpolation conditions and noise moments was proposed by~\cite{rubbens2026computeraidedanalysesstochasticfirstorder}. Our results were informed first by numerical and then analytic solutions to their stochastic PEPs. Interpreting our upper-bound certificate and worst-case construction as matching dual and primal PEP solutions, respectively, establishes tightness of the associated PEP relaxation, whenever $h\in[0,1.5]^N$ and $d\geq N+1$.

Recent works~\cite{bai2025silver,bai2026twostep} have addressed the setting of SGD for stochastic smooth strongly convex minimization. Given an unbiased oracle with bounded variance and independent sampling, the work~\cite{bai2026twostep} characterized the expected squared distance to the minimizer after two steps of a proposed schedule generalizing the silver steps of~\cite{altschuler2025hedging}, finding improvement in the low noise regime.

Finally, we note that further improved guarantees are possible if one allows more general stochastic first-order methods. That is, one may allow averaging, extrapolation, and momentum dynamics beyond just designing algorithms via the selection of stepsizes. We refer readers to~\cite{lan2012}, where such methods are designed, recovering classic accelerated rates as $\sigma$ tends to zero. 

\section{Proof of Upper Bound for any $h\in [0,1.5]^N$} \label{sec:UB}

We prove the claimed upper bound on final expected suboptimality by showing that SGD maintains the nonnegativity of a certain inductive hypothesis $\mathcal H_n$ at each iteration. As a shorthand, denote the first-order values $f_i=f(x_i)$ and $g_i=\nabla f(x_i)$ for any $i\in\{0,\dots,N,\star\}$. Since the point $x_\star$ is a minimizer, $g_\star=0$.

For notational ease, we define the following quantities based on any given schedule of stepsizes
\begin{equation*}
R_n=\sum_{j=n}^{N-1}h_j,\qquad
\alpha_n=\frac{h_n}{2R_n},\qquad \alpha_N=1
\end{equation*}
for $n=0,\dots,N-1$. Throughout our upper bounding proof, we assume $h\in[0,1.5]^N$ but make no assumption on the dimension $d$ of the given problem instance. As an edge case, we set $\alpha_n=0$ when $n<N$ and $R_n=0$. Empty sums and products are zero and one, respectively.

The $L$-smoothness and convexity of $f$ ensure that for any pair of indices $i,j\in\{0,\dots,N,\star\}$ the following is nonnegative
\begin{equation*}
    f_i-f_j-\langle g_j,x_i-x_j\rangle-\tfrac{1}{2L}\|g_i-g_j\|^2 \geq 0.
\end{equation*}
This corresponds to the smooth convex interpolation condition~\cite[Theorem~4]{taylor2017smooth} between points $x_i$ and $x_j$.
Taking an expectation over realizations of the noise vectors $\epsilon$, the expected value of this quantity is also guaranteed to be nonnegative. We denote this quantity by
\begin{equation*}
\mathcal{Q}_{i,j} =\mathbb{E} \left[f_i-f_j-\langle g_j,x_i-x_j\rangle -\tfrac{1}{2L}\|g_i-g_j\|^2\right]\ge0.
\end{equation*}
Our convergence proof will make use of these inequalities with associated weights of
\begin{equation*}
\lambda_{\star,j}
=\alpha_j\prod_{m=0}^{j-1}(1-\alpha_m),\qquad
\lambda_{i,j}=\alpha_i\alpha_j
                 \prod_{m=i+1}^{j-1}(1-\alpha_m)\quad \forall 0\leq i<j\leq N.
\end{equation*}

While SGD runs, we define our inductive hypothesis $\mathcal{H}_n\geq 0$ in terms of the auxiliary vector
\begin{equation}\label{eq:inductive-y}
    y_n=\left[\prod_{m=0}^{n-1}(1-\alpha_m)\right]x_\star
      +\sum_{i=0}^{n-1}\alpha_i
        \prod_{m=i+1}^{n-1}(1-\alpha_m)(x_i-g_i/L).
\end{equation}
Equivalently, this auxiliary sequence can be recursively defined by $y_{n+1}=(1-\alpha_n)y_n + \alpha_n(x_n - g_n/L)$ with $y_0=x_\star$.
For each $n=0,\ldots,N$, we define our inductive hypothesis by the quantity
\begin{equation}\label{eq:inductive-H}
\begin{aligned}
\mathcal H_n= & \frac{LD^2}{2(1+2R_0)}
  +\frac{\sigma^2}{2L}\sum_{j=0}^{n-1}\frac{h_j^2}{1+2R_{j+1}} -\sum_{i=0}^{n-1}\alpha_i\prod_{m=i+1}^{n-1}(1-\alpha_m)
      \mathbb{E}\left[f_i-f_\star-\tfrac{1}{2L}\|g_i\|^2\right]
\\&-\frac{L}{2(1+2R_n)}
      \mathbb{E}\|x_n-y_n\|^2.
\end{aligned}
\end{equation}
The main work in our proof (Lemma~\ref{lem:inductive-step}) will establish nonnegativity of each successive $\mathcal{H}_{n}$ by writing it as a nonnegative combination of $\mathcal{H}_{n-1}$ and quantities assumed to be nonnegative.

Observing that the second line of~\eqref{eq:inductive-H} subtracts a squared term, the inequality $\mathcal{H}_n \geq 0$ implies an upper bound on a nonnegative weighted sum of objective gap-like\footnote{Noting that the descent lemma ensures a deterministic gradient descent step would find an objective value at most $f_i - \frac{1}{2L}\|g_i\|^2$, these terms are each nonnegative, expressing bounds on the objective gap one deterministic step ahead.} terms $\mathbb{E}\left[f_i-f_\star-\tfrac{1}{2L}\|g_i\|^2\right]$. Namely, one can directly read this induction as ensuring
\begin{equation*}
    \sum_{i=0}^{n-1}\alpha_i\prod_{m=i+1}^{n-1}(1-\alpha_m)
      \mathbb{E}\left[f_i-f_\star-\tfrac{1}{2L}\|g_i\|^2\right] \leq \frac{LD^2}{2(1+2R_0)}
  +\frac{\sigma^2}{2L}\sum_{j=0}^{n-1}\frac{h_j^2}{1+2R_{j+1}}.
\end{equation*}
Section~\ref{subsec:moreau-view} provides a more insightful interpretation of the hypothesis $\mathcal{H}_n \geq 0$. There we show that it provides a bound on suboptimality for an associated Moreau envelope at each iteration.

At $n=N$, the claimed sharp upper bound for SGD follows from this hypothesis and the identity
\begin{equation}\label{eq:inductive-terminal}
\begin{aligned}
\frac{LD^2}{2(1+2R_0)}
 +\frac{\sigma^2}{2L}\sum_{j=0}^{N-1}\frac{h_j^2}{1+2R_{j+1}}
 -\mathbb E[f_N-f_\star]
= &\mathcal H_N+\lambda_{\star,N}\mathcal Q_{\star,N} +\sum_{i=0}^{N-1}\lambda_{i,N}\mathcal Q_{i,N}
\\&+\tfrac{L}{2}\mathbb E\left\|
    x_N-y_N
       -\frac{g_N}{L}\right\|^2.
\end{aligned}
\end{equation}
Since each term of the right-hand-side is nonnegative, so is the left-hand-side, providing the claimed rate.
All that remains is to prove that SGD in fact maintains the inductive hypothesis $\mathcal{H}_n\geq 0$. This is done in the following pair of lemmas showing the base case and inductive step.
\begin{lemma}\label{lem:inductive-base}
The initial hypothesis satisfies
\begin{equation*}
\mathcal H_0=\frac{L(D^2-\|x_0-x_\star\|^2)}{2(1+2R_0)}\ge0.
\end{equation*}
\end{lemma}
\begin{proof}
    Equation~\eqref{eq:inductive-y} gives $y_0=x_\star$. Both sums in~\eqref{eq:inductive-H} are empty at $n=0$, giving the
    stated formula. Nonnegativity follows from the assumed $\|x_0-x_\star\|\le D$.
\end{proof}

\begin{lemma}\label{lem:inductive-step}
For every $n=1,\ldots,N$, the following identity holds
\begin{equation}\label{eq:inductive-increment}
\begin{aligned}
\mathcal H_n={}&\mathcal H_{n-1}
 +\lambda_{\star,n-1}\mathcal Q_{\star,n-1}
 +\sum_{i=0}^{n-2}\lambda_{i,n-1}\mathcal Q_{i,n-1}\\
&+\frac{h_{n-1}^2}{2L(1+2R_n)}
       \left(\sigma^2-\mathbb E\|\epsilon_{n-1}\|^2\right)
\\&+\frac{L\alpha_{n-1}(2-h_{n-1}-\alpha_{n-1})}
          {2(1+2R_{n-1})(1+2R_n)}
       \mathbb E\left\|
          x_{n-1}-y_{n-1}
          -\frac{1+2R_{n-1}}{L}g_{n-1}\right\|^2.
\end{aligned}
\end{equation}
Since every added term above is nonnegative when $h\in [0,1.5]^N$,
$\mathcal H_{n-1}\geq0$ implies $\mathcal H_n\geq0$.
\end{lemma}
\begin{proof}
    The definition of $y_n$ and the SGD update give
    \begin{equation}\label{eq:inductive-y-update}
    x_n-y_n=(1-\alpha_{n-1})(x_{n-1}-y_{n-1})
       -\frac{h_{n-1}-\alpha_{n-1}}{L}g_{n-1}
       -\frac{h_{n-1}}{L}\epsilon_{n-1}.
    \end{equation}
    The first two terms are measurable before $\epsilon_{n-1}$ is drawn.
    Conditional zero mean therefore gives
    \begin{equation}\label{eq:inductive-noise}
    \begin{aligned}
    \mathbb E\|x_n-y_n\|^2
       ={}&\mathbb E\left\|
       (1-\alpha_{n-1})(x_{n-1}-y_{n-1})
                  -\frac{h_{n-1}-\alpha_{n-1}}{L}g_{n-1}\right\|^2\\
       &+\frac{h_{n-1}^2}{L^2}\mathbb E\|\epsilon_{n-1}\|^2.
    \end{aligned}
    \end{equation}

    Since $\prod_{m=0}^{k-1}(1-\alpha_m)
    +\sum_{i=0}^{k-1}\alpha_i\prod_{m=i+1}^{k-1}(1-\alpha_m)=1$,
    expanding the interpolation terms using~\eqref{eq:inductive-y} yields,
    for every $k=0,\ldots,N$,
    \begin{align}
    \lambda_{\star,k}\mathcal Q_{\star,k}
      +\sum_{i=0}^{k-1}\lambda_{i,k}\mathcal Q_{i,k}
      =\alpha_k\mathbb{E}\bigg[
      &\sum_{i=0}^{k-1}\alpha_i\prod_{m=i+1}^{k-1}(1-\alpha_m)
        \left(f_i-f_\star-\tfrac{1}{2L}\|g_i\|^2\right)
        \nonumber\\
      &-(f_k-f_\star)
          +\langle g_k,x_k-y_k\rangle
          -\tfrac{1}{2L}\|g_k\|^2
       \bigg]. \label{eq:inductive-interpolation}
    \end{align}
    When subtracting~\eqref{eq:inductive-H} at successive indices, every
    old weight acquires a factor $1-\alpha_{n-1}$. Using
    \eqref{eq:inductive-noise} and~\eqref{eq:inductive-interpolation}
    with $k=n-1$,
    the claimed increment thus reduces to the single quadratic identity
    \begin{equation}\label{eq:inductive-square}
    \begin{aligned}
    &\frac{L}{2(1+2R_{n-1})}
       \|x_{n-1}-y_{n-1}\|^2\\
    &\quad-\frac{L}{2(1+2R_n)}\left\|
       (1-\alpha_{n-1})(x_{n-1}-y_{n-1})
         -\frac{h_{n-1}-\alpha_{n-1}}{L}g_{n-1}\right\|^2\\
    &\quad-\alpha_{n-1}
        \langle g_{n-1},x_{n-1}-y_{n-1}\rangle
        +\frac{\alpha_{n-1}}{L}\|g_{n-1}\|^2\\
    &\qquad=
     \frac{L\alpha_{n-1}(2-h_{n-1}-\alpha_{n-1})}
          {2(1+2R_{n-1})(1+2R_n)}
     \left\|
       x_{n-1}-y_{n-1}
       -\frac{1+2R_{n-1}}{L}g_{n-1}
     \right\|^2.
    \end{aligned}
    \end{equation}
    Expanding the two sides above proves this identity, using that
    $1+2R_n=1+2R_{n-1}-2h_{n-1}$ and
    $(1+2R_{n-1})\alpha_{n-1}=h_{n-1}+\alpha_{n-1}$.
    This proves the inductive identity~\eqref{eq:inductive-increment}.

    Finally, note that $0\le\alpha_{n-1}\le1/2$ and $h_{n-1}\le3/2$.
    These range bounds ensure the nonnegativity of
    $\alpha_{n-1}(2-h_{n-1}-\alpha_{n-1})\ge0$.
    All the other added terms are nonnegative by assumption, altogether
    establishing that the nonnegativity of $\mathcal{H}_n$ is
    inductively maintained.
\end{proof}

\subsection{Moreau Envelope Interpretation of Inductive Hypothesis $\mathcal{H}_n \geq 0$}\label{subsec:moreau-view}

As a final remark on this convergence proof, we note that one can interpret the inductive hypothesis in terms of Moreau envelopes. Namely, for any $\tau>0$, denote the associated Moreau envelope by
\begin{equation*}
    f_\tau(y) = \min_x \left\{f(x) + \frac{1}{2\tau}\|x-y\|^2\right\},
\end{equation*}
with the convention $f_0(y)=f(y)$.

Suppose $\mathcal{H}_n \geq 0$ and consider the values $\tau_n = 2R_n/L$, setting $\tau_N=0$.
Note that the weights in~\eqref{eq:inductive-y} for $y_n$ are nonnegative and sum to one. Applying the smooth convex interpolation condition between an arbitrary point $x$ and each of $x_\star,x_0,\dots,x_{n-1}$, taking the corresponding weighted sum, and completing the square gives
\begin{align*}
    f(x)-f_\star
    \leq{}&
    \sum_{i=0}^{n-1}\alpha_i
        \prod_{m=i+1}^{n-1}(1-\alpha_m)
        \left(f_i-f_\star-\tfrac{1}{2L}\|g_i\|^2\right) +\frac{L}{2}\|x-y_n\|^2\\
    &
    -\frac{L}{2}
        \left\|x-y_n-\frac{\nabla f(x)}{L}\right\|^2.
\end{align*}
Dropping the final nonpositive term provides a quadratic upper bound on $f$. Taking the Moreau envelope of this upper bound with parameter $\tau_n$ therefore gives
\begin{equation*}
    f_{\tau_n}(x_n)-f_\star
    \leq
    \sum_{i=0}^{n-1}\alpha_i
        \prod_{m=i+1}^{n-1}(1-\alpha_m)
        \left(f_i-f_\star-\tfrac{1}{2L}\|g_i\|^2\right)
    +\frac{L}{2(1+L\tau_n)}\|x_n-y_n\|^2.
\end{equation*}
This inequality also holds when $\tau_n=0$ by the stated convention.

With these values, the assumed nonnegativity of $\mathcal{H}_n$ ensures that
\begin{align*}
    \mathbb{E}\left[f_{\tau_n}(x_n)\right]-f_\star
    \leq{}&
    \sum_{i=0}^{n-1}\alpha_i
        \prod_{m=i+1}^{n-1}(1-\alpha_m)
        \mathbb{E}\left[
            f_i-f_\star-\tfrac{1}{2L}\|g_i\|^2
        \right] +\frac{L}{2(1+2R_n)}
        \mathbb{E}\|x_n-y_n\|^2\\
    ={}&
    \frac{LD^2}{2(1+2R_0)}
    +\frac{\sigma^2}{2L}
        \sum_{j=0}^{n-1}\frac{h_j^2}{1+2R_{j+1}}
    -\mathcal{H}_n\\
    \leq{}&
    \frac{LD^2}{2(1+2R_0)}
    +\frac{\sigma^2}{2L}
        \sum_{j=0}^{n-1}\frac{h_j^2}{1+2R_{j+1}}.
\end{align*}
Hence the inductive hypothesis maintains an upper bound on the expected suboptimality of the associated Moreau envelope at the current iterate.

Notice that as $n$ grows, the parameter $\tau_n$ decreases to zero, making the Moreau envelope $f_{\tau_n}$ approach the original objective $f$. At iteration $n=N$, one has $\tau_N=0$, making these actually equal. As a result, this induction ultimately gives the claimed upper bound for SGD as
\begin{equation*}
    \mathbb{E}[f_N]-f_\star
    =
    \mathbb{E}\left[f_{\tau_N}(x_N)\right]-f_\star
    \leq
    \frac{LD^2}{2(1+2R_0)}
    +\frac{\sigma^2}{2L}
        \sum_{j=0}^{N-1}\frac{h_j^2}{1+2R_{j+1}}.
\end{equation*}

\section{Proof of Lower Bound for any $h\in \mathbb{R}^N_+$} \label{sec:LB}

We construct our adversarial problem instance in dimension $d\geq N+1$ as follows:
Denote the $L$-smooth Huber function parameterized by a scalar $a\geq 0$ by
\begin{equation*}
    H_a(z) = \begin{cases}
        \frac{L}{2}|z|^2 & \text{if } |z|\leq a\\
        L a|z| - \frac{L}{2}a^2 & \text{otherwise}.
    \end{cases}
\end{equation*}
Note the degenerate case of $a=0$ has $H_0(z)=0$.
Let $e_0,e_1,\dots,e_N$ denote a collection of $N+1$ orthogonal unit vectors (which exist by the assumption that $d\geq N+1$). Define our constructed problem instance's initialization $x_0 = De_0$.
Define our noise distributions at each iteration $n=1,\dots,N$ by the mutually independent Rademacher distributions
\begin{equation} \label{eq:hard-distributions}
   \epsilon_{n-1} \sim P_{n-1}, \qquad \epsilon_{n-1} = \begin{cases}
        + \sigma e_n & \text{with probability }1/2\\
        - \sigma e_n & \text{with probability }1/2.
    \end{cases}
\end{equation}
Define our objective function by
\begin{equation} \label{eq:hard-function}
    f(x) = \sum_{i=0}^N H_{a_i}(\langle e_i, x\rangle)
\end{equation}
with $a_0 = \frac{D}{1+2 R_0}$ and $a_n = \frac{\sigma h_{n-1}}{
L\left(1+2R_n\right)
}$ for $n=1,\dots,N$.
The following lemma verifies this construction is indeed feasible for any choice of the parameters $L,D,\sigma^2 >0$.
\begin{lemma}
    The problem instance $\mathtt{p}$ defined by $x_0=De_0$, $f$ as in~\eqref{eq:hard-function}, and distributions $P_n$ as in~\eqref{eq:hard-distributions} is feasible, i.e., $\mathtt{p}\in\mathtt{P}(L,D,\sigma^2)$.
\end{lemma}
\begin{proof}
    First, observe that the function $f$ is minimized at the origin $x_\star=0$, having $f_\star=0$. Then, one can verify the needed initial condition as $\|x_0-x_\star\|=\|De_0\|=D$. The needed conditions for each noise distribution $P_n$ of zero-mean and $\sigma^2$-bounded variance follow directly from its definition.
    
    All that remains is to verify that $f$ is convex and $L$-smooth. A simple calculation can verify that the Huber function $H_a$ is convex and $L$-smooth for any $a\geq 0$. Since $f$ is a sum of these Huber functions, it inherits their convexity. Moreover, applying the chain rule, the gradient of $f$ is given by
    \begin{equation*}
        \nabla f(x) = \sum_{i=0}^N H_{a_i}'(\langle e_i,x\rangle)e_i.
    \end{equation*}
    Then $L$-smoothness of $f$ follows from $L$-smoothness of each Huber function as
    \begin{align*}
        \|\nabla f(x)-\nabla f(y)\|^2 &= \|\sum_{i=0}^N (H_{a_i}'(\langle e_i,x\rangle) - H_{a_i}'(\langle e_i,y\rangle))e_i\|^2\\
        &=\sum_{i=0}^N |H_{a_i}'(\langle e_i,x\rangle) - H_{a_i}'(\langle e_i,y\rangle)|^2\\
        &\leq L^2\sum_{i=0}^N |\langle e_i, x-y\rangle|^2\\
        &\leq L^2 \|x-y\|^2.
    \end{align*}
\end{proof}

\subsection{Proof of Lower Bound Half of Theorem~\ref{thm:main-rate}}
    We show that, for any realization of SGD applied to our constructed problem instance, the final objective gap equals the claimed bound. Taking expectations then proves the result.

    Consider any sequence of signs $s_{n}\in\{\pm 1\}$ for $n=1,\dots,N$, providing a realization of errors from $P_{n-1}$ of $\epsilon_{n-1} = s_n \sigma e_n$. The key observation here is that the objective $f$ is separable over the coordinate directions $e_0,\dots,e_N$. So the progression of gradient descent in direction $e_i$ only depends on $\langle e_i, x_n\rangle$ (and no other directions). Thus we can analyze each direction separately.

    First consider the direction $e_0$. Each error $\epsilon_{n-1}$ is orthogonal to this direction, so SGD amounts to applying gradient descent deterministically to minimize $H_{a_0}(\langle e_0, x \rangle)$. This is a standard example in the worst-case analysis of deterministic gradient descent~\cite{drori2014performance,taylor2017smooth,grimmer2025composing}. We claim inductively that every iteration $n=0,\dots,N$ has $\langle e_0, x_n\rangle \geq a_0$. At $n=0$, this is immediate as $D\geq a_0$. Inductively, this follows at some $n\geq 1$ as all prior iterations have derivative in direction $e_0$ of $H'_{a_0}(\langle e_0, x_{i}\rangle) = La_0$. So
    \begin{equation*}
        \langle e_0, x_n\rangle = \langle e_0, x_0\rangle - \sum_{i=0}^{n-1} h_i a_0 = a_0\left(1+2\sum_{i=0}^{N-1} h_i\right) - a_0 \sum_{i=0}^{n-1} h_i = a_0(1+R_0+R_n) \geq a_0
    \end{equation*}
    by noting $\langle e_0, x_0\rangle =D= a_0\left(1+2\sum_{i=0}^{N-1} h_i\right)$ from our choices of $x_0$ and $a_0$.
    
    In particular, the above calculation shows $\langle e_0, x_N\rangle = a_0(1+R_0)$. Hence, the final objective value on this first Huber function component is
    \begin{equation*}
        H_{a_0}(\langle e_0,x_N\rangle) = La_0| a_0(1+R_0)| - \frac{L}{2}a_0^2 = \frac{L}{2}a_0^2(1+2R_0) = \frac{LD^2}{2(1+2R_0)}.
    \end{equation*}

    The same calculations can be repeated on each component $e_1,\dots,e_N$. There, iterations $n=1,\dots,i-1$ have no effect on the $i$th component as $\langle e_i,x_n\rangle=0$. Then iteration $n=i$ will set $\langle e_i,x_i\rangle = -s_{i}\sigma h_{i-1}/L = -s_ia_i(1+2R_i)$. Thereafter, each iteration $n=i+1,\dots,N$ will perform deterministic gradient descent in this coordinate direction. Just as calculated above, gradient descent will always remain on the same (positive or negative) facet of the Huber function $H_{a_i}$, in particular having $\langle e_i,x_n\rangle = -s_ia_i(1+R_i+R_n)$. Continuing with the same calculation, this ensures the final iterate has 
    \begin{equation*}
        H_{a_i}(\langle e_i,x_N\rangle) = \frac{\sigma^2}{2L}\frac{h_{i-1}^2}{1+2 R_{i}},
    \end{equation*}
    regardless of the sampled sign $s_i$. Summing over each coordinate $e_0,\dots,e_{N}$ shows the claimed lower bound is attained for every realization of the signs.

\subsection{Extension to Allow Negative Stepsizes in Lower Bound} \label{sec:negative-steps}
Finally, for any signed schedule $h$, consider the modified nonnegative schedule $h^+=(\max\{h_i,0\})_{i=0}^{N-1}$ and run both schedules with the same noise realization. Each nonconstant Huber component has the same initial or noise displacement. Thereafter, negative steps move away from zero, while positive steps translate both trajectories equally. Inductively, the iterates under $h$ remain on the same facet and at least as far from zero as under $h^+$. Noise at negative steps affects only components with $a_i=0$. Thus, $\mathtt{PEP}(h,L,D,\sigma^2)$ is at least the right-hand-side of~\eqref{eq:pep-value} evaluated at $h^+$.

\paragraph{Acknowledgments.} This work was supported in part by the Air Force Office of Scientific Research under grant FA9550261B093. Benjamin Grimmer was additionally supported as a fellow of the Alfred P. Sloan Foundation.
  
{\small
    \bibliographystyle{plain}
    \bibliography{bibliography}
}

\appendix

\section{Optimal Stepsize Schedule Verifications}

\subsection{Uniqueness and Optimality of Stepsize Schedule~\eqref{eq:optimal-steps}}
\label{app:optimal-stepsizes-verification}

We verify that the stepsizes in~\eqref{eq:optimal-steps} uniquely minimize the lower bound~\eqref{eq:minimized-LB} and attain~\eqref{eq:general-LB}.

First, observe that the objective differentiated in~\eqref{eq:stationarity-stepsizes} is convex: its first term is the reciprocal of a positive affine function, and each remaining term is a quadratic-over-linear function with positive denominator. In fact, the sum of the remaining terms is strictly convex. Equality in its convexity inequality forces the final stepsizes to agree, since its last summand is $h_{N-1}^2$. Working backwards, the denominator of each preceding summand is then fixed, forcing the corresponding stepsizes to agree as well. Thus any stationary point is the unique minimizer.

By the recurrence $c_{i+1}=c_i+c_i^2$, the expression $c_0\prod_{i=0}^{N-1}(1+2c_i)^2$ is continuous and strictly increasing in $c_0\geq0$, starting at zero and tending to infinity. Hence the prescribed choice of $c_0$ exists and is unique, and all of the stepsizes in~\eqref{eq:optimal-steps} are positive.

For the proposed schedule $h^\star$, summing~\eqref{eq:optimal-steps} gives
\begin{align*}
    1+2R_i
    =1+2\sum_{j=i}^{N-1}c_j
    \prod_{k=j+1}^{N-1}(1+2c_k) &=\prod_{k=i}^{N-1}(1+2c_k),
    \qquad \forall i=0,\dots,N.
\end{align*}
Consequently, at $h=h^\star$, one has
$h_i^\star/(1+2R_{i+1})=c_i$, and the defining equation for $c_0$ gives
$L^2D^2/(\sigma^2(1+2R_0)^2)=c_0$.
Differentiating the lower bound in each coordinate $i=0,\dots,N-1$ and evaluating at $h=h^\star$ therefore gives
\begin{align*}
    &\left.
    \frac{\partial}{\partial h_i}
    \left(
        \frac{LD^2}{2(1+2R_0)}
        +\frac{\sigma^2}{2L}\sum_{j=0}^{N-1}
        \frac{h_j^2}{1+2R_{j+1}}
    \right)\right|_{h=h^\star}\\
    &\qquad=
    -\frac{LD^2}{(1+2R_0)^2}
    +\frac{\sigma^2}{L}
    \left(
        \frac{h_i^\star}{1+2R_{i+1}}
        -\sum_{j=0}^{i-1}
        \frac{(h_j^\star)^2}{(1+2R_{j+1})^2}
    \right)\\
    &\qquad=
    \frac{\sigma^2}{L}
    \left(-c_0+c_i-\sum_{j=0}^{i-1}c_j^2\right)
    =0,
\end{align*}
where the last equality follows by summing $c_{j+1}-c_j=c_j^2$.
This verifies~\eqref{eq:stationarity-stepsizes} and establishes the claimed unique minimizer.

To compute its value, the same recurrence and
$1+2R_j=(1+2c_j)(1+2R_{j+1})$ give
\begin{align*}
    \sum_{j=0}^{N-1}\frac{(h_j^\star)^2}{1+2R_{j+1}}
    =\sum_{j=0}^{N-1}
    \left[c_j(1+2R_j)-c_{j+1}(1+2R_{j+1})\right] &=c_0(1+2R_0)-c_N.
\end{align*}
Hence, using the defining equation for $c_0$ once more, we arrive at the claimed rate as
\begin{align*}
    \frac{LD^2}{2(1+2R_0)}
    +\frac{\sigma^2}{2L}\sum_{j=0}^{N-1}
    \frac{(h_j^\star)^2}{1+2R_{j+1}}
    =\frac{LD^2}{1+2R_0}-\frac{\sigma^2}{2L}c_N &=\sigma D\sqrt{c_0}-\frac{\sigma^2}{2L}c_N.
\end{align*}

\subsection{Short Optimal Stepsizes Under Sufficient Variance Bound Condition}
\label{app:optimal-stepsizes-range}

We now verify that $h^\star\in[0,1.5]^N$ whenever
$\sigma^2\geq\frac{2}{3\sqrt{3}}\frac{L^2D^2}{\sqrt N}$.
First, observe that these stepsizes are decreasing. Indeed, for $i=0,\dots,N-2$,
\begin{equation*}
    \frac{h_i^\star}{h_{i+1}^\star}
    =\frac{c_i(1+2c_{i+1})}{c_{i+1}}
    =\frac{1+2c_{i+1}}{1+c_i}>1.
\end{equation*}
Thus it suffices to bound $h_0^\star$.

Since $c_i\geq c_0>0$, the defining equation for $c_0$ gives
\begin{equation*}
    \frac{L^2D^2}{\sigma^2}
    =c_0\prod_{i=0}^{N-1}(1+2c_i)^2
    \geq c_0\left(1+2\sum_{i=0}^{N-1}c_i\right)^2
    \geq4N^2c_0^3.
\end{equation*}
Applying this bound to the first stepsize, we find
\begin{equation*}
    h_0^\star
    =\frac{c_0}{1+2c_0}\prod_{i=0}^{N-1}(1+2c_i)
    =\frac{LD\sqrt{c_0}}{\sigma(1+2c_0)}
    \leq\left(\frac{L^4D^4}{2N\sigma^4}\right)^{1/3}
    \leq\frac32,
\end{equation*}
where the last inequality is precisely the assumed lower bound on $\sigma^2$.
Since every subsequent stepsize in the optimized schedule is positive and no larger than $h_0^\star$, this proves the claimed range.

When $d\geq N+1$, Theorem~\ref{thm:main-rate} then ensures that $h^\star$ attains the lower bound in~\eqref{eq:general-LB}.
Every alternative schedule has a strictly larger lower bound by the uniqueness proved above.
Hence $h^\star$ is also the unique minimax optimal schedule.

\section{Computer and AI Assistance}

The authors used numerical optimization and modern AI tools in preparing this manuscript. The presented worst-case instance and best-possible proof certificate were first identified numerically and then analytically by considering PEP solutions based on the framework of~\cite{rubbens2026computeraidedanalysesstochasticfirstorder}. In this first stage, ChatGPT was a useful assistant in verifying algebra and proposing analytic solutions. 
Once we identified a PEP proof of our main result, the second stage of our work again involved ChatGPT's assistance in producing the (equivalent) inductive proof presented in Section~\ref{sec:UB}. An explicit workflow of using LLMs to assist both in extracting analytic PEP solutions and in deriving associated inductive proofs was recently demonstrated by~\cite{kim2026domainspecificharnessendtoendautomation}.
All mathematical statements, proofs, references, and final editorial decisions in this manuscript are the responsibility of the authors.

\end{document}

%% file: preamble.tex
\usepackage[margin =1in]{geometry}
\usepackage{amssymb,amsthm,amsmath}
\usepackage{enumerate}
\usepackage{hyperref}
\usepackage{cite}
\usepackage[capitalize]{cleveref}
\usepackage[shortlabels]{enumitem}
\usepackage{color}
\usepackage{cancel}
 \usepackage{tikz}
\usepackage{tikz-3dplot}
\usepackage{pgf}
\usepackage{svg}
\usepackage{pgfplots}
\usepackage{import}
\usepackage{svg}
\usepackage{comment}
\usepackage[utf8]{inputenc} % allow utf-8 input
\usepackage[T1]{fontenc}    % use 8-bit T1 fonts
\usepackage{hyperref}       % hyperlinks
\usepackage{url}            % simple URL typesetting
\usepackage{booktabs}       % professional-quality tables
\usepackage{amsfonts}       % blackboard math symbols
\usepackage{nicefrac}       % compact symbols for 1/2, etc.
\usepackage{microtype}
\usepackage{multirow}
\usepackage{xfrac}
\usepackage{todonotes}
\usepackage{titlesec}
\usepackage{caption}
\usepackage{subcaption}

\usepackage{algorithm}
\usepackage{algpseudocode}
\usepackage{array}

\algtext*{EndFor}
\algtext*{EndIf}
\algtext*{EndWhile}
\algtext*{EndProcedure}
\algrenewtext{For}[1]{\textbf{for} #1}
\algrenewtext{While}[1]{\textbf{while} #1}
\algrenewtext{If}[1]{\textbf{if} #1}

\algnewcommand{\Input}{\item[\textbf{Input:}]}
\algnewcommand{\Initialize}{\item[\textbf{Initialize:}]}

\usepackage{varwidth}

\titleformat{\subsubsection}[runin]
{\normalfont\normalsize\bfseries}{\thesubsubsection}{1em}{}

\usepackage{graphicx}

\usepackage{appendix}
\numberwithin{equation}{section}
\theoremstyle{thmstyletwo}%
\newtheorem{thm}{Theorem}[section]
\newtheorem{theorem}[thm]{Theorem}

\newtheorem{lemma}[thm]{Lemma}

\newtheorem{corollary}[thm]{Corollary}

{\theoremstyle{plain}
}

\numberwithin{equation}{section}
\crefname{claim}{claim}{claims}
\Crefname{claim}{Claim}{Claims}
\crefname{lem}{lemma}{lemmas}
\Crefname{lem}{Lemma}{Lemmas}
\crefname{algorithm}{algorithm}{algorithms}
\Crefname{algorithm}{Algorithm}{Algorithms}

\usepackage{mathtools}

\usepackage[T1]{fontenc}
\usepackage{lmodern}
\newlength{\solutionindent}
\newlength{\questionboxrule}
\usepackage[dvipsnames]{xcolor}

\usepackage{hyperref}

\hypersetup{
  pdftitle={
    Exact Theory and Optimal Stepsizes for Stochastic Gradient Descent in Smooth Convex Minimization
  },
  pdfauthor={
    Benjamin Grimmer and Qining Shen
  },
  pdfsubject={
    A characterization of the worst-case performance of any short-step gradient descent method on stochastic smooth convex problems. When the variance bound is moderate, the exactly minimax optimal stepsizes are derived.
  },
  pdfkeywords={
    smooth convex optimization,
    stochastic gradient methods,
    fixed-step first-order methods,
    minimax optimality,
    optimal first-order methods,
    performance estimation problems,
    worst-case analysis,
    first-order oracle complexity
  },
  pdfcreator={LaTeX with hyperref},
  pdfdisplaydoctitle=true,
  unicode=true
}